\documentclass[letterpaper]{article}

\usepackage[letterpaper,left=1in,right=1in,top=1.5in,bottom=1.5in]{geometry}

\newcommand{\myauthor}{Benjamin Antieau and  Gebriele Vezzosi}
\newcommand{\mytitle}{Remarks on the Hochschild-Kostant-Rosenberg theorem in
characteristic p}

\title{A remark on the Hochschild-Kostant-Rosenberg theorem in characteristic $p$}

\author{Benjamin Antieau\footnote{Benjamin Antieau was supported
by NSF Grant DMS-1552766.}  $\,$ and Gabriele Vezzosi}
\date{\today}
 
\usepackage[pdfstartview=FitH,
            pdfauthor={\myauthor},
            pdftitle={\mytitle},
            colorlinks,
            linkcolor=reference,
            citecolor=citation,
            urlcolor=e-mail,
            backref]{hyperref}

\usepackage{amsmath}
\usepackage{amscd}
\usepackage{amsbsy}
\usepackage{amssymb}
\usepackage{verbatim}
\usepackage{eufrak}
\usepackage{eucal}
\usepackage{microtype}
\usepackage{hyperref}
\usepackage{mathrsfs}
\usepackage{amsthm}
\usepackage[all,cmtip]{xy}
\usepackage{tikz}
\usetikzlibrary{matrix,arrows}
\usepackage[abbrev,lite,alphabetic]{amsrefs}
\usepackage{dsfont}

\usepackage{fancyhdr}
\usepackage{makeidx}
\makeindex
\newcommand{\df}[1]{{\bf #1} \index{#1}}

\usepackage{color}
\definecolor{todo}{rgb}{1,0,0}
\definecolor{conditional}{rgb}{0,1,0}
\definecolor{e-mail}{rgb}{0,.40,.80}
\definecolor{reference}{rgb}{.20,.60,.22}
\definecolor{mrnumber}{rgb}{.80,.40,0}
\definecolor{citation}{rgb}{0,.40,.80}

\let\oldmarginpar\marginpar
\renewcommand\marginpar[1]{\-\oldmarginpar[\raggedleft\footnotesize #1]%
{\raggedright\footnotesize #1}}

\newcommand{\Cscr}{\mathcal{C}}
\newcommand{\Dscr}{\mathcal{D}}
\newcommand{\Escr}{\mathcal{E}}
\newcommand{\Fscr}{\mathcal{F}}

\newcommand{\Hscr}{\mathcal{H}}
\newcommand{\Lscr}{\mathcal{L}}
\newcommand{\Mscr}{\mathcal{M}}
\newcommand{\Oscr}{\mathcal{O}}

\newcommand{\Tscr}{\mathcal{T}}

\newcommand{\Xscr}{\mathcal{X}}

\newcommand{\D}{\mathrm{D}}
\newcommand{\E}{\mathrm{E}}

\renewcommand{\H}{\mathrm{H}}
\newcommand{\K}{\mathrm{K}}
\renewcommand{\L}{\mathrm{L}}

\newcommand{\PP}{\mathds{P}}
\newcommand{\QQ}{\mathds{Q}}

\newcommand{\op}{\mathrm{op}}

\newcommand{\ev}{\mathrm{ev}}
\newcommand{\coev}{\mathrm{coev}}

\renewcommand{\geq}{\geqslant}
\renewcommand{\leq}{\leqslant}

\newcommand{\Cat}{\mathrm{Cat}}

\DeclareMathOperator{\Ext}{Ext}

\newcommand{\HH}{\mathrm{HH}}

\newcommand{\Map}{\mathrm{Map}}
\newcommand{\MapSp}{\underline{\mathrm{Ma}}\mathrm{p}}

\newcommand{\ShMap}{\Mscr\mathrm{ap}}
\newcommand{\ShExt}{\Escr\mathrm{xt}}

\newcommand{\Mod}{\mathrm{Mod}}

\newcommand{\Perf}{\mathrm{Perf}}
\newcommand{\Ind}{\mathrm{Ind}}

\DeclareMathOperator{\Spec}{Spec}

\newcommand{\we}{\simeq}
\newcommand{\iso}{\cong}

\theoremstyle{plain}
\newtheorem{theorem}{Theorem}[section]
\newtheorem*{theorem*}{Theorem}
\newtheorem{lemma}[theorem]{Lemma}

\newtheorem{proposition}[theorem]{Proposition}

\newtheorem{corollary}[theorem]{Corollary}

\newtheoremstyle{named}{}{}{\itshape}{}{\bfseries}{.}{.5em}{#1 \thmnote{#3}}
\theoremstyle{named}

\theoremstyle{definition}

\newtheorem{example}[theorem]{Example}
\newtheorem{question}[theorem]{Question}

\newtheorem{remark}[theorem]{Remark}

\begin{document}

\maketitle

\begin{abstract}
    \noindent
    We prove a Hochschild-Kostant-Rosenberg decomposition theorem for smooth
    proper schemes $X$ in characteristic $p$ when $\dim X\leq p$. The best known previous result
    of this kind, due to Yekutieli, required $\dim X<p$. Yekutieli's result
    follows from the observation that the denominators appearing in the classical proof of HKR do not divide $p$ when $\dim
    X<p$. Our extension to $\dim X=p$ requires a homological
    fact: the Hochschild homology of a smooth proper scheme is self-dual.


    \paragraph{Mathematics Subject Classification 2010.}
    \href{http://www.ams.org/mathscinet/msc/msc2010.html?t=14Fxx&btn=Current}{14F10},
    \href{http://www.ams.org/mathscinet/msc/msc2010.html?t=16Exx&btn=Current}{16E40},
    \href{http://www.ams.org/mathscinet/msc/msc2010.html?t=19Dxx&btn=Current}{19D55}.
\end{abstract}

\section{Introduction}

The classical Hochschild-Kostant-Rosenberg theorem
of~\cite{hochschild-kostant-rosenberg} states that if $k$ is a commutative ring
and $R$ is a smooth commutative $k$-algebra, then there is a natural isomorphism
$\Omega_{R/k}^*\iso\HH_*(R/k)$ of graded-commutative $R$-algebras. In fact, \emph{when $k$
is a $\QQ$-algebra}, 
this isomorphism lifts to the level of complexes,
giving a natural quasi-isomorphism $\HH(R/k)\we\bigoplus_t\Omega_{R/k}^t[t]$. Here,
$\HH(R/k)$ denotes Hochschild chains, any one of the natural complexes
computing Hochschild homology, and
$\bigoplus_t\Omega_{R/k}^t[t]$ is viewed as a complex with zero differential. In
particular, we see that the Hochschild chains are naturally formal for smooth
affine schemes over characteristic $0$ fields.

Swan extends the HKR decomposition in~\cite{swan}*{Corollary~2.6} to smooth quasi-projective schemes in
characteristic $0$ using related work of
Gerstenhaber-Schack~\cite{gerstenhaber-schack}.\footnote{Often these authors are
more concerned with Hochschild cohomology, or Hochschild cochains. The
results typically dualize without trouble.} Swan's work implies in particular that there are
canonical decompositions $$\HH_n(X/k)\iso\bigoplus_{t-s=n}\H^s(X,\Omega_{X/k}^t)$$ for all
$n$. Using the fact that Hochschild
homology for all commutative $k$-algebras is determined by its values on smooth
$k$-algebras, the HKR decomposition was extended to all commutative
$k$-algebras (still in characteristic $0$) in work of Buchweitz and Flenner~\cite{buchweitz-flenner},
Schuhmacher~\cite{schuhmacher}, and To\"en and Vezzosi~\cite{tv-simplicial}.
One obtains a natural decomposition
\begin{equation}\label{eq:tv}
    \HH(R/k)\we\bigoplus_{t\geq 0}\L\Lambda^t\left(\L_{R/k}\right)[t],
\end{equation}
where $\L_{R/k}$
is the cotangent complex and $\L\Lambda^t$ is the derived
functor of the $t$th exterior power. Note that Hochschild homology is not
typically formal in the non-smooth case but that the differentials are all
supported in the cotangent complex direction and not in the exterior algebra
direction. Of course, there is a global version of this decomposition for schemes as well.

Much less is known in characteristic $p$. The main result to date is due to Yekutieli
who proves in~\cite{yekutieli}*{Theorem~4.8} that there is a natural HKR quasi-isomorphism in characteristic $p$ for
smooth schemes of dimension less than $p$. To make this precise, we let $\underline{\HH}_X$ denote
the complex of quasi-coherent sheaves on $\Oscr_X$ computing Hochschild homology. In other
words, if $U\subseteq X$ is an affine open subscheme, then
$\underline{\HH}_X(U)\we\HH(U/k)$. One model for this complex is
$\Delta_X^*(\Oscr_{\Delta_X})$,
where $\Delta_X:X\rightarrow X\times_kX$ is the diagonal morphism, $\Delta^*_X$ is the
derived pullback functor,\footnote{Here and elsewhere we mean derived functors unless
specified otherwise.} and $\Oscr_{\Delta_X}$ is
the structure sheaf of the diagonal inside $X\times_kX$.

Yekutieli proved that if $k$ is a commutative ring and if
$X\rightarrow\Spec k$ is a smooth morphism of relative dimension $d$, then
$$\underline{\HH}_X\we\bigoplus_{t=0}^d\Omega_{X/k}^t[t]$$ as complexes of sheaves on $X$ under the assumption
that $d!$ is invertible in $k$. Again, $\bigoplus\Omega^t_{X/k}[t]$ is viewed
as a complex of sheaves with zero differential. This implies that
there are natural isomorphisms
$\HH_n(X/k)\iso\bigoplus_{s-t=n}\H^s(X,\Omega^t_{X/k})$ for all $n$ under the assumption that $d!$ is invertible in $k$.

\begin{question}\label{q:hkr}
    Is there an HKR theorem for smooth schemes in characteristic $p$?
\end{question}

There are multiple ways that this question might be answered. The explicit
isomorphism constructed in~\cite{hochschild-kostant-rosenberg} does not extend in general to
smooth schemes $X$ in
characteristic $p$ (unless $\dim X<p$). However, if $R$ is \emph{any} smooth smooth $k$-algebra, no
matter what the dimension, then there is an HKR-like
quasi-isomorphism.\footnote{Using simplicial commutative rings, one may show that the
quasi-isomorphism in the proposition can be chosen to respect multiplicative structures by
showing that $\HH(R/k)$ is weakly equivalent to the free simplicial commutative
$R$-algebra on $\Omega^1_{R/k}[1]$. This is a way of saying that $\HH(R/k)$ is formal as a
simplicial commutative ring.}

\begin{proposition}\label{prop:formality}
    Let $k$ be a commutative ring and let $R$ be a smooth commutative $k$-algebra. Then,
    there is a quasi-isomorphism $\bigoplus_t\Omega_{R/k}^t[t]\we\HH(R/k)$.
\end{proposition}

\begin{proof}
    This is a special case of a more general fact: any complex with projective
    homology is formal. A quasi-isomorphism is obtained by choosing maps from
    the homology into the complex, which can always be done thanks to
    projectivity.
\end{proof}

The crucial point in the proof above is that we must choose a lift to get our map
$\Omega^1_{R/k}[1]\rightarrow\HH(R/k)$. Without characteristic or dimension assumptions, it
is not known how to make this choice natural in $R$, which would be necessary to extend the
result to schemes.\footnote{The issue is the same as in the gap in the proof of the main
theorem of~\cite{mccarthy-minasian}. In the proof of Theorem 6.1, they argue that their HKR
theorem is true \'etale locally and hence it is true globally. However, for this argument to
work, they must give a globally defined map.}

Given the Hochschild-Kostant-Rosenberg theorem, it is not hard to prove that the Hochschild
homology sheaves $\Hscr\Hscr_t$ are canonically isomorphic to $\Omega_X^t$ for smooth
schemes $X$. In particular, there is a natural local-global spectral sequence
$$\E_2^{s,t}=\H^s(X,\Omega_X^t)\Rightarrow\HH_{t-s}(X).$$
We say that the \df{weak HKR theorem holds for $X$ over $k$} if this spectral sequence degenerates at the
$\E_2$-page. We say that the \df{strong HKR theorem holds for $X$ over $k$} if $\underline{\HH}_X$ is
formal as a complex of sheaves; i.e., if there is a quasi-isomorphism
$$\bigoplus_t\Omega^t_{X/k}[t]\we\underline{\HH}_X.$$ Evidently, the strong HKR theorem for
$X$ implies the weak HKR theorem for $X$.

Summarizing the past work in this language, we know by~\cite{hochschild-kostant-rosenberg}
and Proposition~\ref{prop:formality} that the strong HKR theorem holds for
$X\rightarrow\Spec k$ smooth affine. By Yekutieli ~\cite{yekutieli}*{Theorem~4.8}, we also know that the strong HKR theorem
holds for $X$ over $k$ when $X\rightarrow\Spec k$ is smooth of relative dimension $d$ and
$d!$ is invertible in $k$. When $k$ is a field of characteristic $p$, this condition is the
same as asking for $\dim X<p$.

Our main results establish the strong HKR theorem for
smooth proper schemes $X$ with $\dim X\leq p$ over characteristic $p$ fields.

\begin{theorem}\label{thm:main}
    Suppose that $k$ is a field of characteristic positive $p$ and that $X$ is a smooth proper
    $k$-scheme of dimension at most $p$. Then, the weak HKR theorem holds for $X$.
    Specifically, if $\dim X=p$, then for each $n$ there is a canonical short exact sequence
    $$0\rightarrow\H^{p-n}(X,\Omega_X^p)\rightarrow\HH_n(X/k)\rightarrow\bigoplus_{t=0}^{p-1}\H^{t-n}(X,\Omega^{t})\rightarrow 0,$$
    which is split (but possibly not canonically split).
\end{theorem}

The main idea in the proof of this theorem is the use of self-duality for $\HH(X/k)$ in the local-global spectral sequence 
$\H^{s}(X,\Omega^t_{X/k})\Rightarrow\HH_{t-s}(X/k)$. A similar move occurs
in~\cite{deligne-illusie}, where the authors use the compatibility of Serre duality with
the Cartier isomorphism and a trace argument to boost their result on degeneration of the
Hodge-de Rham spectral sequence for smooth $X/k$ with $\dim X<p$ to smooth $X/k$ with $\dim
X\leq p$. See~\cite{deligne-illusie}*{Corollaire~2.3}.

\begin{example}
    The theorem in particular implies the HKR decomposition for smooth projective surfaces
    in characteristic $2$. This answers the explicit form of Question~\ref{q:hkr} asked by
    Daniel Pomerleano on {\ttfamily mathoverflow} in~\cite{mo:16960}.
\end{example}

The theorem has the following corollary.

\begin{corollary}\label{cor:main}
    Suppose that $k$ is a field of positive characteristic $p$ and that $X$ is a smooth proper
    $k$-scheme of dimension at most $p$. Then, the strong HKR theorem holds for $X$.
\end{corollary}

We briefly review Hochschild homology in Section~\ref{sec:hochschild} and we give the proofs of
the theorem and corollary as well as some more examples in Section~\ref{sec:proofs}.

\begin{remark}\label{rem:lifting}
    It is interesting to wonder at the connection between the weak HKR theorem and degeneration
    of the Hodge-de Rham spectral sequence. Given that there is no liftability hypothesis in
    Theorem~\ref{thm:main}, it is not at all clear if or when there should be a relation.
    We can however make the following weak statement. If $X$ is a smooth scheme over a characteristic $p$ field $k$ that lifts
    to a smooth proper map $\Xscr\rightarrow\Spec R$ where $R$ is any discrete valuation
    ring with characteristic zero fraction field and residue field $k$, and if the
    cohomology groups $\H^s(X,\Omega_{\Xscr/R}^t)$ are $p$-torsion-free for all $s,t$, then the
    weak HKR theorem holds for $X$ over $t$. Indeed, the differentials must all be
    $p$-torsion by the HKR theorem in characteristic $0$, but
    the groups are $p$-torsion-free. Hence, the differentials all vanish for the spectral
    sequence of $\Xscr$ over $R$. But, the hypotheses also imply that the spectral sequence of
    $X$ over $k$ is the mod $p$ reduction of that for $\Xscr$ over $R$.
\end{remark}

\begin{example}\label{ex:ci}
    By lifting to characteristic $0$ and~\cite{sga7-2}*{Expos\'e~XI, Th\'eor\`eme 1.5}, the
    remark implies that the weak HKR theorem holds for smooth complete
    intersections in projective space. Despite the sparsity of the local-global
    spectral sequence in that case, this statement is not entirely trivial. For example, for
    such a $5$-fold $X$ in characteristic $2$, the differential
    $d_2:\H^2(X,\Omega_{X/k}^3)\rightarrow\H^4(X,\Omega_X^4)$ could, for all we
    would know otherwise, be non-zero. In fact, we can also see that this
    differential must
    be zero by using the duality arguments used in the proof of
    Theorem~\ref{thm:main} the pullback map from the local-global spectral
    sequence for $\PP^n$ to $X$.
\end{example}

Finally, let us say a word about multiplicative structures.
In~\cite{tv-simplicial}, To\"en and Vezzosi prove the strong HKR theorem in
characteristic zero and show that the equivalence of~\eqref{eq:tv} is naturally
$S^1$-equivariant and multiplicative. The proof of Corollary~\ref{cor:main}
does in particular induce an equivalence between the sheaf of simplicial
commutative rings $\underline{\HH}_X$ and the
sheaf of free simplicial commutative rings on $\Omega^1_{X/k}$ and there is a
corresponding formality result upon taking global sections. We emphasize
again that we do not know how to make this $S^1$-equivariant or functorial in $X$, or even if that
is possible.

\paragraph{Acknowledgments.} Many people have wondered about the existence of
HKR-type results in characteristic $p$ and we have talked to many people about
the subject, including Damien Calaque, David Gepner, Ezra Getzler, Michael
Gr\"ochenig, M\'arton Hablicsek, Akhil Mathew, and Boris Tsygan.

BA was first asked this question by Yank\i\, Lekili in the fall of 2014 and
thanks him for many conversations about this and many other topics. Special
thanks also go to Bhargav Bhatt who suggested to BA the idea of bringing duality into
play, which had already been considered by GV. The
resulting conversations led directly to the present note.

\section{Hochschild homology}\label{sec:hochschild}

Let $k$ be a commutative ring and let $\Cat_k$ denote the $\infty$-category of
small idempotent complete pretriangulated $k$-linear dg categories up to derived Morita invariance. 
The $\infty$-category $\Cat_k$ is
equivalent to the Dwyer-Kan localization of the category of small $k$-linear dg
categories (or even small \emph{flat} $k$-linear dg categories)
at the class of $k$-linear derived Morita equivalences.
Objects of $\Cat_k$ will be called \df{$k$-linear categories} for
short.\footnote{Equivalently, $\Cat_k$ is the $\infty$-category of small
idempotent complete $k$-linear stable $\infty$-categories.}

Let $\Dscr(k)$ denote the derived $\infty$-category of $k$.
The homotopy category of $\Dscr(k)$ is a triangulated
category equivalent to $\D(k)$, the derived category of $k$.
We let $\HH(-/k):\Cat_k\rightarrow\Dscr(k)$ be the Hochschild homology functor as
studied in~\cite{keller-cyclic}.\footnote{To be precise, $\HH(-/k)$ is the derived functor of
Keller's mixed complex construction (obtained by taking flat resolutions of dg categories).
Since the mixed complex construction inverts derived Morita equivalences
(by~\cite{keller-cyclic}*{Section~1.5}) it descends to a map
$\Cat_k\rightarrow\Dscr(k)$
of $\infty$-categories.} The following proposition is well-known
(see~\cite{cisinski-tabuada}*{Example~8.9}
or~\cite{bgt2}*{Corollary~6.9}); we give a sketch of the proof for the reader's
convenience.

\begin{proposition}
    The functor $\HH(-/k):\Cat_k\rightarrow\Dscr(k)$ is symmetric monoidal.
\end{proposition}

\begin{proof}
    We will indicate how the functor from small flat $k$-linear dg categories to chain complexes
    over $k$ given by taking the mixed complex as in~\cite{keller-cyclic}*{Section~1.3} is
    symmetric monoidal. Indeed, given a small flat $k$-linear dg category $\Cscr$, the
    mixed complex $C(\Cscr)$ of~\cite{keller-cyclic} is the total complex of the bicomplex
    associated to the simplicial chain complex $C_\bullet(\Cscr)$ with $n$th term
    $$\bigoplus_{(x_n,\cdots,x_0)}\Cscr(x_n,x_0)\otimes_k\Cscr(x_{n-1},x_n)\otimes_k\Cscr(x_{n-2},x_{n-1})\otimes_k\cdots\otimes_k\Cscr(x_1,x_2)\otimes_k\Cscr(x_0,x_1),$$
    where the sum is over all $(n+1)$-tuples of objects of $\Cscr$. The differentials are as
    usual for the cyclic bar complex. (See~\cite{keller-cyclic}*{Section 1.3}.)
    Given a second small flat $k$-linear dg category, we see that there is a natural map of
    simplicial chain complexes
    $C_\bullet(\Cscr)\otimes_k C_\bullet(\Dscr)\rightarrow C_\bullet(\Cscr\otimes_k\Dscr)$
    obtained using the symmetric monoidal structure on chain complexes. This map is
    a level-wise quasi-isomorphism.
    Taking the associated total complexes (and using the
    shuffle product), we obtain a natural quasi-isomorphism $C(\Cscr)\otimes_k C(\Dscr)\rightarrow
    C(\Cscr\otimes_k\Dscr)$ giving a symmetric monoidal structure. Compare with the statement of the Hochschild homological
    Eilenberg-Zilber theorem (see~\cite{loday}*{Theorem~4.2.5}).
    Since $\Cat_k$ is the localization of small $k$-linear dg categories at the
    $k$-linear derived Morita equivalences, and since the mixed complex functor
    $C$ inverts $k$-linear derived Morita equivalences
    by~\cite{keller-cyclic}*{Section 1.5}, the proposition follows.
\end{proof}

Let $\Cscr$ be a $k$-linear category. In this case, $\Ind(\Cscr)$ is dualizable
in $\Mod_{\Dscr(k)}(\Pr^\L)$. Let
$\Dscr(k)\xrightarrow{\coev_{\Cscr}}\Ind(\Cscr)\otimes\Ind(\Cscr^\op)$ be the
coevaluation map and let
$\Ind(\Cscr^{\op})\otimes\Ind(\Cscr)\xrightarrow{\ev_{\Cscr}}\Dscr(k)$ be the
evaluation map. We say that $\Cscr$ is \df{smooth} if the coevaluation map
preserves compact objects. We say that $\Cscr$ is \df{proper} if the evaluation
map preserves compact objects. Note that $\Cscr$ is smooth and proper if and
only if it is dualizable in $\Cat_k$. See~\cite{toen-vaquie}*{Definition~2.4}
or~\cite{ag}*{Lemma~3.7}. If $\Cscr$ is dualizable, then the dual is equivalent to the
opposite dg category $\Cscr^{\op}$.

\begin{corollary}\label{cor:perfection}
    If $\Cscr$ is a smooth and proper $k$-linear category, then $\HH(\Cscr/k)$
    is perfect as a complex over $k$.
\end{corollary}

\begin{proof}
    In this case, $\Cscr$ is dualizable in $\Cat_k$. Since symmetric monoidal
    functors preserve dualizable objects, the corollary follows at once from
    the fact that the dualizable objects of $\Dscr(k)$ are precisely the
    perfect complexes.
\end{proof}

\begin{remark}
    It is important to use $\Cat_k$ as opposed to the more familiar $\infty$-category
    $\mathrm{dgAlg}_k$ of dg $k$-algebras in the corollary because
    smooth and proper $k$-linear categories are typically \emph{not} dualizable
    in $\mathrm{dgAlg}_k$.
\end{remark}

\section{Proofs}\label{sec:proofs}

We give the proofs of Theorem~\ref{thm:main} and Corollary~\ref{cor:main} in this section
after a couple more preliminaries. The next statement is well-known over fields (see for
example~\cite{orlov}). The proof over a general commutative ring $k$ is the same, so we omit
it.

\begin{proposition}
    Let $k$ be a commutative ring. If $X\rightarrow\Spec k$ is a smooth and
    proper morphism of schemes, then $\Perf_X$ is dualizable as a $k$-linear category.
\end{proposition}

\begin{corollary}\label{cor:hhdual}
    In the situation above, $\HH(X/k)$ is dualizable as a complex over $k$
    with dual the Hochschild homology of the smooth and proper $k$-linear
    category $\left(\Perf_X\right)^{\op}$.
\end{corollary}

\begin{proof}
    This follows from Corollary~\ref{cor:perfection} taking into account that the dualizable
    objects of $\Dscr(k)$ are precisely the perfect complexes.
\end{proof}

The $k$-linear categories of the form $\Perf_X$ are very special: they are equivalent to
their own opposites. Below, if $\Escr$ and $\Fscr$ are
complexes of quasi-coherent sheaves on $X$, we write $\MapSp_X(\Escr,\Fscr)$
for the mapping spectrum, a spectrum whose homotopy groups
$\pi_i\MapSp_X(\Escr,\Fscr)$ are given by
$\Ext^{-i}_X(\Escr,\Fscr)$. Similarly, $\ShMap_X(\Escr,\Fscr)$ is the complex
of quasi-coherent sheaves on $X$ whose homotopy sheaves are
$\pi_i\ShMap_X(\Escr,\Fscr)\iso\ShExt^{-i}_X(\Escr,\Fscr)$.

\begin{lemma}\label{lem:selfdual}
    Let $X$ be a $k$-scheme and let $\Lscr$ be a line bundle on $X$.
    Then, there is an equivalence $\Perf_X\we\Perf_X^{\op}$ of $k$-linear categories
    obtained by $\Fscr\mapsto\ShMap_X(\Fscr,\Lscr)$.
\end{lemma}

\begin{proof}
    The claim can be checked Zariski locally on $X$, hence for affine schemes,
    where it is obvious.
\end{proof}

\begin{corollary}\label{cor:selfdual}
    If $X$ is a smooth and proper $k$-scheme, then the complex of $k$-modules $\HH(X/k)$ is self-dual.
    That is, the evaluation map is a non-degenerate pairing $\HH(X/k)\otimes_k\HH(X/k)\rightarrow k$.
\end{corollary}

\begin{proof}
    This is an immediate consequence of Corollary~\ref{cor:hhdual} and
    Lemma~\ref{lem:selfdual}.
\end{proof}

The existence of such a pairing has been observed in the literature before. For example, it
appears implicitly in Shklyarov~\cite{shklyarov}*{Theorem~1.4} and is studied by
C\u{a}ld\u{a}raru and Willerton in~\cite{caldararu-willerton}, who call it the Mukai pairing.
It also occurs in the preprint~\cite{tv-bloch} of To\"en and Vezzosi.

Our point of departure is to pair the pairing of Corollary \ref{cor:selfdual} with the convergent local-global spectral
sequence
\begin{equation}\label{eq:lg}
    \E_2^{s,t}=\H^{s}(X,\Omega^t_{X/k})\Rightarrow\HH_{t-s}(X/k).
\end{equation}
We will see that this quickly leads to a proof of the main theorem.
We need one more lemma, which is also implied by~\cite{yekutieli}*{Theorem~4.8}.

\begin{lemma}\label{lem:decomposition}
    Let $k$ be a ring in which $p$ acts nilpotently and let $X\rightarrow\Spec
    k$ be a smooth proper morphism.
    Let $\tau_{\leq p-1}\underline{\HH}_X$ denote the $(p-1)$st truncation of
    $\underline{\HH}_X$. Then, there is a natural quasi-isomorphism $$\tau_{\leq
    p-1}\underline{\HH}_X\we\bigoplus_{t=0}^{p-1}\Omega_{X/k}^t[t].$$
\end{lemma}

\begin{proof}
    In general, for any smooth scheme, there is a natural map of complexes of sheaves
    $\underline{\HH}_X\rightarrow\Omega_{X/k}^t[t]$ constructed in~\cite{loday}*{Section
    1.3}. On smooth affine schemes $X=\Spec R$, this map can be described as taking a
    Hochschild chain $r_0\otimes r_1\otimes\cdots\otimes r_t$ to the
    differential $t$-form $r_0dr_1\cdots dr_t$. This map is \emph{not} the map arising
    in the Hochschild-Kostant-Rosenberg theorem. Rather, there is an isomorphism
    $\Omega_{X/k}^t\rightarrow\Hscr\Hscr_t$ of sheaves. The induced composition
    $\Omega_{X/k}^t\rightarrow\Omega_{X/k}^t$ is multiplication by
    $t!$ by~\cite{loday}*{Proposition~1.3.16}. But, this implies that the induced map
    $\tau_{\leq p-1}\HH_X\rightarrow\bigoplus_{t=0}^{p-1}\Omega_{X/k}^t[t]$ is a
    quasi-isomorphism under the present hypotheses.
\end{proof}

We are now ready to give the proof of the main theorem.

\begin{proof}[Proof of Theorem~\ref{thm:main}]
    If $\dim X<p$, then the theorem follows from Lemma~\ref{lem:decomposition}
    or~\cite{yekutieli}*{Theorem~4.8}. So, assume that $\dim X=p$.
    It follows from Lemma~\ref{lem:decomposition} that the only possibly non-zero differentials in the
    local-global spectral sequence~\eqref{eq:lg} are those hitting
    $\E_r^{s,p}=\H^s(X,\Omega_{X/k}^p)$. These
    groups are only possibly non-zero for $0\leq s\leq p$, and they contribute to
    Hochschild homology in degrees $p,p-1,\ldots,0$. The differential $d_r$ has bidegree
    $(r,r-1)$, which lowers the total degree by $1$. In particular, the only terms that
    support a non-zero differential out must in particular have total degree $1,\ldots,p$.
    By Serre duality,
    $\H^s(X,\Omega_{X/k}^t)\iso\H^{p-s}(X,\Omega_{X/k}^{p-t})^\vee$, the
    $k$-linear dual of $\H^{p-s}(X,\Omega^{p-t}_{X/k})$. In particular, at the
    $\E_2$-page, the sum of the dimensions of total degree $a\geq 1$ is equal to the sum of
    the dimensions of total degree $-a$. If some term of degree $a\geq 1$ supports a
    differential to $\E_r^{s,p}$, then $\dim_k\HH_a(X/k)<\dim_k\HH_{-a}(X/k)$, which
    contradicts Corollary~\ref{cor:selfdual}. It follows that there are no non-zero
    differentials, so that $X$ satisfies the weak HKR theorem over $k$. The last statement
    follows from the fact that
    $$\Omega^p_{X/k}[p]\rightarrow\underline{\HH}(X/k)\rightarrow\bigoplus_{t=0}^{p-1}\Omega_{X/k}^t[t]$$
    is a cofiber sequence and the lack of differentials in~\eqref{eq:lg} implies we
    get short exact sequences in global sections, as desired.
\end{proof}

Now, we prove Corollary \ref{cor:main}.

\begin{proof}[Proof of Corollary~\ref{cor:main}]
    We can again assume that $\dim X=p$.
    By Lemma~\ref{lem:decomposition}, it is enough to construct a map $\underline{\HH}_X\rightarrow\omega_{X/k}[p]$ such that
    the composition
    $\omega_{X/k}[p]\rightarrow\underline{\HH}_X\rightarrow\omega_{X/k}[p]$ is
    the identity. In other words, we are interested in the restriction map
    $$\MapSp_X(\underline{\HH}_X,\omega_{X/k}[p])\rightarrow\MapSp_X(\omega_{X/k}[p],\omega_{X/k}[p]).$$
    This is the map on global sections of the map of (sheaves of) mapping complexes
    $$\ShMap_X(\underline{\HH}_X,\omega_{X/k}[p])\rightarrow\ShMap_X(\omega_{X/k}[p],\omega_{X/k}[p]),$$
    and there is a corresponding map of local-global spectral sequences which is a
    surjection on the $\E_2$-pages since $\Hscr\Hscr_p\iso\omega_{X/k}$. We will be done if
    we show that the local-global spectral sequence
    $$\E_2^{s,t}=\H^s(X,\pi_t\ShMap_X(\underline{\HH}_X,\omega_{X/k}[p]))\Rightarrow\pi_{t-s}\MapSp_X(\underline{\HH}_X,\omega_{X/k}[p])$$
    collapses at the $\E_2$-page. However,
    $\pi_t\ShMap_X(\underline{\HH}_X,\omega_{X/k}[p])\iso\Omega_{X/k}^t$ using
    the natural isomorphisms
    $$\Tscr_{X/k}^{p-t}\otimes\omega_{X/k}\iso\Omega_{X/k}^t,$$ where
    $\Tscr_{X/k}^{p-t}$ denotes the $(p-t)$th exterior power of the tangent
    bundle of $X$ over $k$; By Grothendieck duality,
    $\Map_X(\underline{\HH}_X,\omega_{X/k}[p])\we\HH(X/k)^\vee$, the $k$-dual
    of $\HH(X/k)$. But, by Corollary~\ref{cor:selfdual},
    $\HH(X/k)^\vee\we\HH(X/k)$. By a dimension count and using
    Theorem~\ref{thm:main}, we see that the local-global spectral sequence computing
    $\pi_*\MapSp_X(\underline{\HH}_X,\omega_{X/k}[p])$ does indeed collapse.
\end{proof}

Given the proof above, we ask the following question.

\begin{question}
    Is $\underline{\HH}_X\we\ShMap_X(\underline{\HH}_X,\omega_{X/k}[d])$ when $X$ is a
    smooth proper $k$-scheme of dimension $d$?
\end{question}

\begin{lemma}
    Let $X$ be a smooth proper $d$-dimensional scheme over a field $k$ of characteristic $p>0$.
    If the weak HKR theorem holds for $X$, then
    $\underline{\HH}_X\we\ShMap_X(\underline{\HH}_X,\omega_{X/k}[d])$.
\end{lemma}

\begin{proof}
    The proof above of Corollary~\ref{cor:main} applies equally well here to show
    that there is a map $\underline{\HH}_X\rightarrow\omega_{X/k}[d]$ such that
    the composition
    $\omega_{X/k}[d]\rightarrow\underline{\HH}_X\rightarrow\omega_{X/k}[d]$ is the
    identity. Now, consider the composition
    $\underline{\HH}_X\otimes_{\Oscr_X}\underline{\HH}_X\rightarrow\underline{\HH}_X\rightarrow\omega_{X/k}[d]$
    induced by the multiplicative structure on $\underline{\HH}_X$.
    The reader can check that the adjoint map
    $\underline{\HH}_X\rightarrow\ShMap_X(\underline{\HH}_X,\omega_{X/k}[d])$
    is an equivalence using that the multiplication on the homotopy sheaves of
    $\underline{\HH}_X$ agrees with the exterior power multiplication on $\Omega_{X/k}^*$.
\end{proof}

We end the paper with a brief connection to algebraic $K$-theory.

\begin{proposition}
    Let $X$ be a smooth projective $3$-fold over a field $k$ of characteristic
    $2$.  If the image of the first Chern class map
    $c_1:\K_0(X)\rightarrow\H^1(X,\Omega_X^1)$ generates $\H^1(X,\Omega_X^1)$
    as a $k$-vector space, then the weak HKR theorem holds for $X$.
\end{proposition}

\begin{proof}
    It is not hard using that we can make $\underline{\HH}_X$ into an
    $\Oscr_X$-algebra to see that there are no non-zero differentials leaving
    the terms $\H^s(X,\Oscr_X)$. The only other possible differential that hits
    a class of negative degree is
    $d_2:\H^1(X,\Omega^1_X)\rightarrow\H^3(X,\Omega^2_X)$. This differential vanishes
    as all of the classes must be permanent thanks to the hypothesis and the
    trace map $\K_0(X)\rightarrow\HH_0(X/k)$. Now, all remaining classes
    involve terms of total degrees $-2$, $-1$, or $0$. In any case, if one of
    these differentials is non-zero, then the self-duality of $\HH(X/k)$ is
    violated, just as in the proof of Theorem~\ref{thm:main}.
\end{proof}

This hypothesis is satisfied for smooth complete intersections in $\PP^4$
by Deligne's theorem that these have the same Hodge numbers as their
characteristic $0$ counterparts~\cite{sga7-2}*{Expos\'e~XI}.
In this case the result also follows as a special case of Example~\ref{ex:ci}.


%


\addcontentsline{toc}{section}{References}

\begin{bibdiv}
\begin{biblist}



\bib{ag}{article}{
    author={Antieau, Benjamin},
    author={Gepner, David},
    title={Brauer groups and \'etale cohomology in derived
    algebraic
    geometry},
    journal={Geom. Topol.},
    volume={18},
    date={2014},
    number={2},
    pages={1149--1244},
    issn={1465-3060},
}



%
%
%
%

%
%
%
%
%
%
%




\bib{bgt2}{article}{
    author={Blumberg, Andrew J.},
    author={Gepner, David},
    author={Tabuada, Gon\c{c}alo},
    title={Uniqueness of the multiplicative cyclotomic trace},
    journal={Adv. Math.},
    volume={260},
    date={2014},
    pages={191--232},
    issn={0001-8708},
}

%
%
%
%
%

%
%
%
%
%
%
%

\bib{buchweitz-flenner}{article}{
   author={Buchweitz, Ragnar-Olaf},
   author={Flenner, Hubert},
   title={The global decomposition theorem for Hochschild (co-)homology of
   singular spaces via the Atiyah-Chern character},
   journal={Adv. Math.},
   volume={217},
   date={2008},
   number={1},
   pages={243--281},
   issn={0001-8708},
}

\bib{caldararu-willerton}{article}{
    author={C\u ald\u araru, Andrei},
    author={Willerton, Simon},
    title={The Mukai pairing. I. A categorical approach},
    journal={New York J. Math.},
    volume={16},
    date={2010},
    pages={61--98},
    issn={1076-9803},
}

%
%

\bib{cisinski-tabuada}{article}{
    author={Cisinski, Denis-Charles},
    author={Tabuada, Gon\c{c}alo},
    title={Symmetric monoidal structure on non-commutative motives},
    journal={J. K-Theory},
    volume={9},
    date={2012},
    number={2},
    pages={201--268},
    issn={1865-2433},
}

%
%

\bib{sga7-2}{book}{
    author={Deligne, Pierre}
    author={Katz, Nicholas}
    title={Groupes de monodromie en g\'eom\'etrie alg\'ebrique. II},
    series={Lecture Notes in Mathematics, Vol. 340},
    note={S\'eminaire de G\'eom\'etrie Alg\'ebrique du Bois-Marie
    1967--1969 (SGA 7 II); Dirig\'e par P. Deligne et N. Katz},
    publisher={Springer-Verlag, Berlin-New York},
    date={1973},
    pages={x+438},
}

\bib{deligne-illusie}{article}{
    author={Deligne, Pierre},
    author={Illusie, Luc},
    title={Rel\`evements modulo $p^2$ et d\'ecomposition du complexe de de
    Rham},
    journal={Invent. Math.},
    volume={89},
    date={1987},
    number={2},
    pages={247--270},
    issn={0020-9910},
}

%
%
%
%
%


\bib{gerstenhaber-schack}{article}{
    author={Gerstenhaber, Murray},
    author={Schack, S. D.},
    title={A Hodge-type decomposition for commutative algebra
    cohomology},
    journal={J. Pure Appl. Algebra},
    volume={48},
    date={1987},
    number={3},
    pages={229--247},
    issn={0022-4049},
}

%
%
%
%
%
%
%


\bib{hochschild-kostant-rosenberg}{article}{
    author={Hochschild, G.},
    author={Kostant, Bertram},
    author={Rosenberg, Alex},
    title={Differential forms on regular affine algebras},
    journal={Trans. Amer. Math. Soc.},
    volume={102},
    date={1962},
    pages={383--408},
    issn={0002-9947},
}
%
%
%
%
%
%

\bib{keller-cyclic}{article}{
    author={Keller, Bernhard},
    title={On the cyclic homology of exact categories},
    journal={J. Pure Appl. Algebra},
    volume={136},
    date={1999},
    number={1},
    pages={1--56},
    issn={0022-4049},
}

%
%
%

\bib{loday}{book}{
    author={Loday, Jean-Louis},
    title={Cyclic homology},
    series={Grundlehren der Mathematischen Wissenschaften},
    volume={301},
    note={Appendix E by Mar\'\i a O. Ronco},
    publisher={Springer-Verlag, Berlin},
    date={1992},
    pages={xviii+454},
    isbn={3-540-53339-7},
}

%
%
%


\bib{mccarthy-minasian}{article}{
    author={McCarthy, Randy},
    author={Minasian, Vahagn},
    title={HKR theorem for smooth $S$-algebras},
    journal={J. Pure Appl. Algebra},
    volume={185},
    date={2003},
    number={1-3},
    pages={239--258},
    issn={0022-4049},
}

%
%
%


%

\bib{orlov}{article}{
    author={Orlov, Dmitri},
    title={Smooth and proper noncommutative schemes and gluing of DG
    categories},
    journal={Adv. Math.},
    volume={302},
    date={2016},
    pages={59--105},
    issn={0001-8708},
}


\bib{mo:16960}{misc}{    
    title={Hochschild Kostant Rosenberg theorem for varieties in positive
    characteristic?},    
    author={Pomerleano, Daniel},
    note={URL: https://mathoverflow.net/q/16960 (version: 2010-03-03)},    
    eprint={https://mathoverflow.net/q/16960},    
    organization={MathOverflow}  
}

%
%
%
%
%
%

\bib{schuhmacher}{article}{
    author={Schuhmacher, Frank},
    title={Hochschild cohomology of complex spaces and Noetherian
    schemes},
    journal={Homology Homotopy Appl.},
    volume={6},
    date={2004},
    number={1},
    pages={299--340},
    issn={1532-0081},
}

\bib{shklyarov}{article}{
    author={Shklyarov, D.},
    title={Hirzebruch-Riemann-Roch-type formula for DG algebras},
    journal={Proc. Lond. Math. Soc. (3)},
    volume={106},
    date={2013},
    number={1},
    pages={1--32},
    issn={0024-6115},
}

\bib{swan}{article}{
    author={Swan, Richard G.},
    title={Hochschild cohomology of quasiprojective schemes},
    journal={J. Pure Appl. Algebra},
    volume={110},
    date={1996},
    number={1},
    pages={57--80},
    issn={0022-4049},
}

%

\bib{toen-vaquie}{article}{
    author={To\"en, Bertrand},
    author={Vaqui\'e, Michel},
    title={Moduli of objects in dg-categories},
    journal={Ann. Sci. \'Ecole Norm. Sup. (4)},
    volume={40},
    date={2007},
    number={3},
    pages={387--444},
    issn={0012-9593},
}

\bib{tv-simplicial}{article}{
    author={To\"en, Bertrand},
    author={Vezzosi, Gabriele},
    title={Alg\`ebres simpliciales $S^1$-\'equivariantes, th\'eorie de
    de Rham et
    th\'eor\`emes HKR multiplicatifs},
    journal={Compos. Math.},
    volume={147},
    date={2011},
    number={6},
    pages={1979--2000},
    issn={0010-437X},
}

\bib{tv-bloch}{article}{
    author={To\"en, Bertrand},
    author={Vezzosi, Gabriele},
    title = {G\'eom\'etrie non-commutative, formule des traces et conducteur de Bloch},
    journal = {ArXiv e-prints},
    eprint = {http://arxiv.org/abs/1701.00455},
    year = {2017},
}

%
%
%
%
%
%
%



%
%
%



\bib{yekutieli}{article}{
   author={Yekutieli, Amnon},
   title={The continuous Hochschild cochain complex of a scheme},
   journal={Canad. J. Math.},
   volume={54},
   date={2002},
   number={6},
   pages={1319--1337},
   issn={0008-414X},
}
		


\end{biblist}
\end{bibdiv}
\end{document}